\documentclass[reqno]{amsart}
\usepackage{amsfonts,amsmath,amsthm,amssymb,mathrsfs}
\usepackage{color}
\usepackage{amsmath,amssymb,amsfonts,bm}
\usepackage{graphicx}
\usepackage{newunicodechar}
\usepackage{xcolor}
\usepackage{geometry}
\numberwithin{equation}{section}
\usepackage[pdfstartview=FitH,colorlinks=true]{hyperref}
\hypersetup{urlcolor=red, linkcolor=black,citecolor=red}
\allowdisplaybreaks

\theoremstyle{plain}

\newtheorem{thm}{Theorem}[section]

\newtheorem{prop}[thm]{Proposition}
\newtheorem{cor}[thm]{Corollary}
\newtheorem{lemma}[thm]{Lemma}
\newtheorem{remark}[thm]{Remark}

\newcommand{\Div}{\operatorname{div}}

\begin{document}
\title[Backward Uniqueness]
{Backward Uniqueness for Coupled Ultraparabolic Operators and an Application to Jerk-Driven Control Models}

\author[X.-D.Cao \& C.-J. Xu \& Yan Xu]
{Xiao-Dong Cao, Chao-Jiang Xu, Yan Xu}

\address{Xiao-Dong Cao, Chao-Jiang Xu
\newline\indent
School of Mathematics and Key Laboratory of Mathematical MIIT,
\newline\indent
Nanjing University of Aeronautics and Astronautics, Nanjing 210016, China
\newline\indent
Yan Xu
\newline\indent
Department of Mathematical Sciences, Tsinghua University, Beijing 100084, China
}
\email{caoxiaodong@nuaa.edu.cn; xuchaojiang@nuaa.edu.cn; xu-y@mail.tsinghua.edu.cn}

\date{\today}



\begin{abstract}
We prove backward uniqueness for a class of ultraparabolic operators with coupled linear drift. The main difficulty is that the Fourier transform in the degenerate variables turns the coupled drift into a transport operator in the dual frequency variables, so the classical Littlewood--Paley Carleman argument does not apply directly. We overcome this by introducing an invariant frequency variable and establishing a frequency-localized Carleman estimate adapted to the transport structure. The result gives a partial answer to the question of W. Wang and L. Zhang $\left[ \emph {Methods Appl. Anal.}, \ 20 \ (1) \ (2013) \ 79-88 \right]$ for constant coupled drift, with diffusion and lower-order coefficients depending on time and the diffusive variables. As an application, for a jerk-driven control model, we prove backward uniqueness for the equation describing the position, velocity, acceleration, or jerk error: under bounded lower-order coefficients, zero final error in $L^2$ implies zero error at all earlier times.
\end{abstract}

\maketitle

\section{Introduction}
\noindent
{\bf 1.1 Background.}
Ultraparabolic operators arise in several areas of analysis and applied mathematics, including kinetic theory \cite{Des, Kol-1, Villani-1, o-3, MX-1}, boundary layer theory \cite{Xin-Zhang}, stochastic processes \cite{Sonin-1}, financial mathematics \cite{Bar-1, BS, fi-1} and control theory; {\em see Section \ref{jerk section} below for a detailed discussion based on jerk models \cite{control-MM}}. From a mathematical point of view, ultraparabolic operators play an important role in the study of degenerate parabolic operators and often possess hypoelliptic structures. In contrast to uniformly parabolic operators, diffusion occurs only in part of the variables. The regularity theory for such operators has also led to useful analytic tools in kinetic theory. For related developments, we refer to \cite{Imbert, Mouhot01}. For H\"older regularity, a breakthrough De Giorgi-type result for equations with rough coefficients was obtained by W. Wang and L. Zhang \cite{Wang-Zhang01}. An alternative proof based on the diffusion variable averaging was later provided by F. Golse, C. Imbert, C. Mouhot and A. F. Vasseur \cite{GIMV}.

In a series of works, L. Escauriaza, G. Seregin and V. \v{S}ver\'ak \cite{ESS, ESS2} showed that the parabolic backward uniqueness property can be used to establish smoothness of Leray-Hopf solutions in $L^{\infty}_t L^3_x$. Their work has also inspired analogous critical regularity results for geometric flows, such as C. Wang's work \cite{C.W} on the heat flow of harmonic maps. For  backward uniqueness of the heat operator in different domains, there are many related works; we refer to \cite{ESS, WW} and references therein. For general parabolic operators, J. Wu and L. Zhang proved the backward uniqueness property both in a half space \cite{W-Z01} and in the whole space \cite{W-Z02} under certain assumptions on the leading coefficients of the operator. In fact, the backward uniqueness property of the general second-order parabolic operator is closely related to the Landis--Oleinik conjecture; see \cite{TAN-Z, W-Z03} for further details.

A standard approach to backward uniqueness is based on Carleman estimates, which originate from Carleman’s work \cite{Car} on unique continuation for elliptic equations. This topic has been extensively studied, and several general frameworks have been developed. We refer to \cite{Yamamoto} for further background.

\bigskip
\noindent
{\bf 1.2 Previous work and the main obstruction.}
For ultraparabolic operators, Carleman estimates are much more subtle than in the uniformly parabolic case. After the usual conjugation by an exponential weight, the main task is to obtain a weighted estimate in which all lower-order and error terms can be absorbed. In the uniformly parabolic case, this is helped by the fact that the second-order part differentiates in every spatial direction; first-order terms are then usually treated as bounded perturbations, or at least as terms controlled by the left-hand side of the Carleman inequality. For ultraparabolic operators, this simple picture breaks down. The diffusion is present only in some variables, and the other directions are seen only through the transport part and the corresponding commutators. At the same time, the drift is not a bounded lower-order term. In typical examples its coefficients grow linearly in space, such as $x\partial_y$. This interaction between the partial diffusion and unbounded drift is a major difficulty in proving backward uniqueness.

In the work of W. Wang and L. Zhang \cite{Wang-Zhang000}, backward uniqueness was studied for the following ultraparabolic operator:
\begin{equation}\label{special case}
\mathcal {L} = \partial_t + \sum_{\ell = 1}^m x_\ell \partial_{y_\ell} + \sum_{i, k = 1}^n \partial_{x_i} (a_{ik}(t, x) \partial_{x_k}).
\end{equation}

Their argument relies on a Littlewood--Paley decomposition in the degenerate variables. For the model \eqref{special case}, after taking the Fourier transform in $y$, the drift $\sum_{\ell=1}^m x_\ell \partial_{y_\ell}$ becomes
$i\sum_{\ell=1}^m x_\ell \eta_\ell$. In particular, no derivative in the Fourier variable $\eta$ appears. 
This absence of frequency transport is one of the key structural features used in the Carleman argument of \cite{Wang-Zhang000}.

This feature is lost for coupled linear drifts. For instance, consider
\begin{equation}\label{L1}
\mathcal{L}_1 = \partial_t + x \partial_y + y \partial_z + \partial_x^2.
\end{equation}
After taking the Fourier transform in $(y,z)$, the transport part becomes $\partial_t + i\eta x - \zeta \partial_\eta$.
The additional term $-\zeta\partial_\eta$ differentiates in the Fourier variable $\eta$, and hence transports the $y$-frequency. 
This frequency transport prevents a direct adaptation of the argument in \cite{Wang-Zhang000}.

This frequency-transport obstruction explains why the problem raised in {\em Remark 1.2} of \cite{Wang-Zhang000} is substantially more difficult than the model treated there.

\bigskip
\noindent
{\bf 1.3 Statement of the main result.}
In this paper, we study backward uniqueness for ultraparabolic operators with coupled linear drift, under structural assumptions on the diffusion matrix and lower-order coefficients analogous to those in \cite{Wang-Zhang000}. More precisely, we consider the backward uniqueness problem for the following ultraparabolic operators: 
\begin{equation}\label{ultraparabolic operator}
P = \partial_t + \left( B_1 v + B_2 w \right) \cdot \nabla_w + \Div_v \left( A(t, v) \nabla_v \right), \quad (t, v, w) \in (0, T) \times \mathbb{R}^m \times \mathbb{R}^n,  
\end{equation}
where $T > 0$ is fixed, $m, n \ge 1$, the matrices $B_1 \in \mathbb{R}^{n \times m}$ and $B_2 \in \mathbb{R}^{n \times n}$ satisfy the following condition: 
\begin{equation}\label{condition-B1B2}
\operatorname{rank} \left[ B_1, B_2B_1, \dots, B_2^{n - 1} B_1 \right] = n. 
\end{equation}
This condition means that the diffusive directions, together with their iterated commutators with the linear drift, generate all degenerate directions. The diffusion matrix $A(t, v) = (a_{ij}(t, v))_{1 \le i, j \le m}$ is symmetric and uniformly elliptic in the diffusive variables: 
\begin{equation}\label{condition-A}
a_{ij} (t, v) = a_{ji} (t, v) \in C^{0, 1}((0, T) \times \mathbb{R}^m); \qquad \lambda^{-1} |\xi|^2 \le \sum_{i, j = 1}^m a_{ij}(t, v) \xi_i \xi_j \le \lambda |\xi|^2, 
\end{equation} 
for all $(t, v) \in (0, T) \times \mathbb{R}^m$ and $\xi \in \mathbb{R}^m$, where $\lambda > 1$ is a fixed constant. Moreover, we assume that $u$ satisfies: 
\begin{equation}\label{condition-2}
\left\{
\begin{aligned}
&P u = c(t, v) u + d(t, v) \cdot \nabla_v u, \quad (t, v, w) \in (0, T) \times \mathbb{R}^m \times \mathbb{R}^n, \\
&u(0, v, w) = 0, \quad (v, w) \in \mathbb{R}^m \times \mathbb{R}^n, \\
&u, \nabla_v u, \left( \partial_t + (B_1 v + B_2 w) \cdot \nabla_w \right) u \in L^2((0, T) \times \mathbb{R}^{m} \times \mathbb{R}^n). 
\end{aligned}
\right.
\end{equation}
Here $d(t, v) = \left( d_1(t, v), \dots, d_m(t, v) \right)$ is an $\mathbb{R}^m$-valued function and the coefficients $c(t, v), d(t, v)$ are bounded functions. For simplicity, throughout the paper we assume the following uniform bounds:
\begin{equation}\label{condition-3}
|\partial_t a_{ij} (t, v)| + |\nabla_v a_{ij}(t, v)| + |c(t, v)| + |d(t, v)| \le \lambda, \quad 1 \le i, j \le m, \quad \lambda > 1, 
\end{equation}
where $\displaystyle |d(t, v)| = \left( \sum_{j = 1}^m |d_j (t, v)|^2 \right)^{1 / 2}$. 

The main result of this paper is stated as follows:
\begin{thm}\label{main}
Let $P$ be the operator in \eqref{ultraparabolic operator}. Suppose that $u$ satisfies \eqref{condition-2} in the distributional sense specified above, and that $B_1, B_2, A, c, d$ satisfy \eqref{condition-B1B2}, \eqref{condition-A} and \eqref{condition-3}. Then
$$
u \equiv 0, \quad a.e. \ (t, v, w) \in (0, T) \times \mathbb{R}^m \times \mathbb{R}^n.
$$
\end{thm}

\begin{remark}
The solution $u$ in Theorem \ref{main} is not assumed to be smooth. The equation in \eqref{condition-2} holds in the sense of distributions, and the initial condition is interpreted as a zero $L^2(\mathbb R^{m} \times \mathbb{R}^n)$-trace at $t=0$. Thus the theorem applies to functions satisfying the $L^2$-regularity assumptions stated in the last line of \eqref{condition-2}. No additional classical differentiability is required.
\end{remark}

\begin{remark}
Here we present two illustrative examples of the model \eqref{ultraparabolic operator}. Firstly, if we take $m = 1$, let
$$
B_1= \left( 1, 0, \cdots, 0 \right)^T \in \mathbb{R}^{n \times 1}, \qquad
B_2=
\begin{pmatrix}
1 &0&0&\cdots&0&1\\
1&1&0&\cdots&0&0\\
0&1&1&\cdots&0&0\\
\vdots&\ddots&\ddots&\ddots&\vdots&\vdots\\
0&\cdots&0&1&1&0\\
0&\cdots&0&0&1&1
\end{pmatrix} \in \mathbb{R}^{n \times n},
\qquad 
A(t, v) = a_1 (t, v).
$$
With these choices, the operator \eqref{ultraparabolic operator} becomes 
$$
P = \partial_t + (v + w_1 + w_n) \partial_{w_1} + (w_1 + w_2) \partial_{w_2} + \dots + (w_{n - 1} + w_n) \partial_{w_n} + \partial_v (a_1 (t, v) \partial_v).
$$
For these choices, the rank condition \eqref{condition-B1B2} follows from a direct computation of $[B_1,B_2B_1,\ldots,B_2^{n-1}B_1]$. Secondly, the operator \eqref{L1} is obtained by taking
$$
B_1 = (1, 0)^T, \qquad
B_2 = 
\begin{pmatrix}
0&0\\
1&0
\end{pmatrix}, \qquad
A(t, v) = 1.
$$
\end{remark}

\bigskip
\noindent
{\bf 1.4 Idea of the proof and the adapted frequency cutoff.} As explained above, the usual Littlewood--Paley decomposition in the original degenerate frequency variables is not adapted to the coupled drift considered here. The key point is therefore to identify an invariant of the frequency transport and to construct the cutoff in this invariant variable. We outline the main idea of the proof in five steps.

{\em Step 1. (Fourier transform)} After taking the Fourier transform in the degenerate variable $w$: 
$$
\mathscr{F}_w u (t, v, \eta) = (2\pi)^{-\frac n2} \int_{\mathbb{R}^n} e^{-iw\cdot \eta} u(t, v, w) dw. 
$$
Writing $\widehat u=\mathscr F_wu$,  we obtain: 
\begin{align*}
\mathscr{F}_w (Pu) = \left[ \partial_t -(B_2^T\eta)\cdot\nabla_\eta
+i(B_1^T\eta)\cdot v +\operatorname{div}_v(A(t,v)\nabla_v) -\operatorname{tr}B_2 \right] \widehat u.
\end{align*}
Here $B_j^T$ denotes the transpose of $B_j$, and $\operatorname{tr}B_2$ denotes the trace of $B_2$. The constant zeroth-order term $-\operatorname{tr}B_2$ is harmless and will be absorbed into the lower-order part.

{\em Step 2. (Choosing an invariant frequency variable)} To overcome the difficulty caused by the term $-(B_2^T \eta) \cdot \nabla_\eta$, we look for a frequency variable that is invariant under the transport operator 
$$
\partial_t - (B_2^T \eta) \cdot \nabla_\eta.
$$
We set 
\begin{equation}\label{rho}
\rho = e^{tB_2^T} \eta. 
\end{equation}
A direct computation gives: 
$$
\left( \partial_t - (B_2^T \eta) \cdot \nabla_\eta \right)\rho = 0. 
$$
Thus $\rho$ is invariant under the frequency transport. Therefore the low-frequency cutoff is imposed in the $\rho$-variable rather than in the original frequency variable $\eta$.

{\em Step 3. (Removing the $\eta$-derivative)} We define the transform
\begin{equation}\label{T}
(\mathcal{T}u)(t, v, \rho) = e^{-\frac12 t \operatorname{tr} B_2} \mathscr{F}_w u(t, v, e^{-tB_2^T} \rho). 
\end{equation}
The factor $e^{-\frac12 t\operatorname{tr}B_2}$ makes $T$ unitary after the change of variables $\rho=e^{tB_2^T}\eta$. Under this transform, the operator $P$ becomes 
\begin{equation}\label{def-01}
\widetilde{P}_\rho : = \partial_t + i (B_1^T e^{-t B_2^T \rho}) \cdot v + \Div_v (A(t, v) \nabla_v) - \frac12 (\operatorname{tr} B_2).
\end{equation}
Thus the transport in the degenerate Fourier variables has been removed. The remaining difficulty is the time-dependent linear term 
$$
i(B_1^Te^{-tB_2^T}\rho)\cdot v,
$$
which has to be controlled in the Carleman estimate.

{\em Step 4. (Low-frequency region)} The cutoff must be imposed in the invariant variable $\rho$, rather than in the original Fourier variable $\eta$. Fix a local time interval $(0,t_2)\subset (0,T)$ and a constant $b$ such that $0<b\leq t_2$. For $R>0$, define
\begin{equation}\label{low}
\mathcal{U}_R := \left\{ \rho \in \mathbb{R}^n : \sup_{0 < t < t_2} (t + b)^3 |B_1^T e^{-t B_2^T} \rho|^2 \le R \right\}.
\end{equation}
The definition of $\mathcal U_R$ is chosen to fit the condition \eqref{condi-1} of Proposition \ref{local Carleman}. For every $\rho\in\mathcal U_R$, the left-hand side of \eqref{condi-1} is at most $R$. Thus, if $\alpha\geq\alpha_0$ and $R\leq\varepsilon_0\alpha$, Proposition \ref{local Carleman} applies to every frequency in $\mathcal U_R$. The quantity $B_1^T e^{-tB_2^T}\rho$ is used because it is the coefficient of the linear term
$$
i(B_1^T e^{-tB_2^T}\rho)\cdot v
$$
in $\widetilde P_\rho$ of \eqref{def-01}. In the proof of Proposition \ref{local Carleman}, Lemma \ref{lemma02} shows that this term can be absorbed precisely under the bound of \eqref{condi-1}. Therefore the low-frequency region is defined through $B_1^T e^{-tB_2^T}\rho$, rather than through $|\rho|$.

The rank condition \eqref{condition-B1B2} plays an essential role here. Besides expressing the hypoelliptic structure of the operator, it gives a quantitative lower bound for the invariant frequencies. More precisely, Lemma \ref{condition-B1B2-proof} shows that there exists $c_2>0$ such that
$$
\sup_{0<t<t_2}(t+b)^3\left|B_1^T e^{-tB_2^T}\rho\right|^2
\ge c_2|\rho|^2,\quad \rho\in\mathbb R^n.
$$
Hence, if $\rho\in \mathcal U_R$, then $|\rho|^2\le R/c_2$, so each $\mathcal U_R$ is bounded. On the other hand, every fixed $\rho\in\mathbb R^n$ belongs to $\mathcal U_R$ once $R$ is large enough. Therefore
$$
\bigcup_{R>0}U_R=\mathbb R^n.
$$

{\em Step 5. (Low-frequency projection)} To localize the transformed equation to the above low-frequency region, we define 
\begin{equation}\label{S}
\mathcal{S}_R := \mathcal{T}^{-1} \mathbf 1_{\mathcal{U}_R} \mathcal{T}, 
\end{equation}
where ${\bf 1}_{\mathcal U_R}$ denotes the characteristic function of $\mathcal U_R$. The meaning of this definition is simple: after applying the transform $\mathcal T$, the operator $\mathcal S_R$ keeps only the part with $\rho\in \mathcal U_R$ and sets the part with $\rho\notin \mathcal U_R$ to zero. Equivalently,
$$
\mathcal T(\mathcal S_Ru)={\bf 1}_{\mathcal U_R}\mathcal Tu.
$$
Since $\bigcup_{R>0}\mathcal U_R=\mathbb R^n$, proving $\mathcal S_Ru=0$ for every $R>0$ allows us to let $R\to\infty$ and recover the full solution. Thus the cutoff is made in the invariant frequency variable $\rho$, rather than in the original Fourier variable $\eta$. This is the reason why $\mathcal S_R$ is compatible with the frequency transport generated by the coupled drift.

Moreover, since $A$, $c$ and $d$ are independent of $w$, applying $\mathcal S_R$ does not create commutator terms involving these coefficients. Hence the equation for $\mathcal S_Ru$ retains the same lower-order structure.

\bigskip
\noindent
{\bf 1.5 Weak Carleman estimate and structure of the paper.}
The proof of Theorem \ref{main} is based on the following weak Carleman estimate. 
\begin{thm}\label{global Carleman}
Let constants $C\geq 1$, $C_*\geq 1$, 
$\alpha_0>1$, and $c_0\in(0,1)$ sufficiently small, depending only on $\lambda,m,n,B_1,B_2$.  Let $R>0$, $\alpha,t_2,b$ satisfy
$$
\alpha\geq \alpha_0+C_*R, \qquad 0<t_2\leq \min\{c_0\lambda^{-2},T\}, \qquad 0<b\leq t_2.
$$
Then for any $g(t, v, w) \in C^\infty_c ((0, t_2) \times \mathbb{R}^m \times \mathbb{R}^n)$, we have 
\begin{equation}\label{Carleman-1}
\begin{aligned}
&\quad \int_0^{t_2} \int_{\mathbb{R}^m} \int_{\mathbb{R}^n} (t + b)^{-2\alpha} |\nabla_v \mathcal{S}_R g|^2 dvdwdt + \alpha\int_0^{t_2} \int_{\mathbb{R}^m} \int_{\mathbb{R}^n} (t + b)^{-2\alpha - 1} |\mathcal{S}_R g|^2 dvdwdt \\
&\le C \int_0^{t_2} \int_{\mathbb{R}^m} \int_{\mathbb{R}^n} (t + b)^{-2\alpha + 1} |P\mathcal{S}_R g|^2 dvdwdt.
\end{aligned}
\end{equation}
\end{thm}
\begin{remark}
We call \eqref{Carleman-1} a weak Carleman estimate because it controls only the low-frequency component $\mathcal S_Rg$. The parameter $R$ is fixed first, and the Carleman parameter $\alpha$ is then sent to infinity. After obtaining $\mathcal S_Ru=0$ for every $R > 0$, we use
$
\bigcup_{R>0}\mathcal U_R=\mathbb R^n
$.
As $\mathcal S_R=\mathcal T^{-1}\mathbf 1_{\mathcal U_R}\mathcal{T}$, this implies $\mathcal{T}u=0$ almost everywhere, and therefore $u=0$.
\end{remark}

Theorem \ref{global Carleman} is derived from the following Carleman estimate in the diffusive variables, with the invariant frequency $\rho$ fixed.
\begin{prop}\label{local Carleman}
Let constants $C\ge 1, \alpha_0>1$, and $c_0, \varepsilon_0 \in(0,1)$ sufficiently small, depending only on $\lambda,m,n,B_1,B_2$. Let $\alpha,t_2,b$ satisfy:
$$
\alpha\geq \alpha_0, \qquad 0<t_2\leq \min\{c_0\lambda^{-2},T\}, \qquad 0<b\leq t_2.
$$
Fix $\rho \in \mathbb{R}^n$. If 
\begin{equation}\label{condi-1}
\sup_{0 < t < t_2} (t + b)^3 |B_1^T e^{-t B_2^T} \rho|^2 \le \varepsilon_0 \alpha.
\end{equation}
Then for any $h(t, v) \in C_c^\infty ((0, t_2) \times \mathbb{R}^m)$, we have 
\begin{equation}\label{Carleman-2}
\int_0^{t_2} \int_{\mathbb{R}^m} (t + b)^{-2\alpha} |\nabla_v h|^2 dvdt + \alpha \int_0^{t_2} \int_{\mathbb{R}^m} (t + b)^{-2\alpha - 1} |h|^2 dvdt \le C \int_0^{t_2} \int_{\mathbb{R}^m} (t + b)^{-2\alpha + 1} |\widetilde{P}_\rho h|^2 dvdt.
\end{equation}
\end{prop}

\begin{remark}
The assumption that $A(t, v), c(t, v), d(t, v)$ are independent of the degenerate variable w is essential in the present argument. Indeed, the low-frequency cutoff $\mathcal S_R$ is defined through the invariant Fourier variable associated with the coupled drift. Therefore $\mathcal S_R$ commutes with $A, c, d$ only because these coefficients do not depend on $w$. If the coefficients were allowed to depend on the degenerate variables, additional commutator terms would appear, and the present Carleman estimate would no longer close. Thus Theorem \ref{main} should be viewed as a backward uniqueness result for ultraparabolic operators with constant coupled linear drift and coefficients depending on the time and diffusive variables. The fully variable-coefficient problem remains open; see \cite{Zhang000}. Nevertheless, the invariant-frequency cutoff introduced here may provide a useful tool for studying such more general operators.
\end{remark}

The rest of the paper is organized as follows. In Section \ref{jerk section}, we discuss a control-theoretic application based on jerk-driven control models and derive a qualitative backward uniqueness consequence from Theorem \ref{main}. Sections \ref{proof of Carleman} and \ref{proof of BU} are devoted to the proof of the weak Carleman estimate and the proof of the main theorem respectively. Section \ref{appendix} contains the proofs of the technical lemmas used in Section \ref{proof of Carleman}.

\section{A control-theoretic application: jerk-driven control models}\label{jerk section}
In this section, we give an application from control theory that leads to the operator studied in this paper. In a jerk-driven control model, the motion is described not only by position and velocity, but also by acceleration and jerk, where jerk means the time derivative of acceleration. This produces a chain structure: jerk affects acceleration, acceleration affects velocity, and velocity affects position. When random perturbations or modeling errors enter through the jerk variable, the corresponding linear error equation has diffusion only in the jerk direction, while the other variables are connected through this chain. We will write such an equation for one chosen error component, for example the position, velocity, acceleration, or jerk error. After reversing time, this equation becomes an operator of the form covered by Theorem \ref{main}. Therefore, if the whole error function is zero at the final time in the $L^2$ sense, then it must be zero at all earlier times.

\subsection{Background in control theory}
The mathematical modeling of target tracking processes is one of the central problems in control theory.
In aerospace target tracking and vehicle motion modeling, the acceleration of highly maneuvering targets may change significantly over a short period of time.
For such problems, lower-order motion models involving only the position, velocity, and acceleration of the tracked object are often not accurate enough.
For this reason, the so-called jerk variable, namely the time derivative of acceleration, was introduced in \cite{control-MM} as a new state variable to describe the rate of change of acceleration. The corresponding model is called the jerk model.

At the same time, Ghazaei--Robertsson--Johansson \cite{Ghazaei} explicitly reformulated robotic trajectory generation as a controller design problem and derived a minimum-jerk optimal controller by means of the HJB equation. The jerk-controlled robot path tracking problem studied by Palleschi et al. \cite{Palleschi} is even more direct: they take jerk itself as the control input and study the time-optimal path tracking problem under velocity, acceleration, and jerk constraints. Therefore, the jerk model is not merely an estimation model in target tracking; it can also be viewed as a typical source of higher-order state chains and higher-order control constraints in control theory.

We now look more closely at the jerk model. For clarity of the subsequent discussion, we consider the one-dimensional case. Let
$$
\mathcal Q=\text{position},\qquad
\mathcal V=\text{velocity},\qquad
\mathcal A=\text{acceleration},\qquad
J=\text{jerk}.
$$
Since jerk is the derivative of acceleration, we have
\begin{equation}\label{chain}
\frac{d\mathcal A}{ds}=J,\qquad
\frac{d\mathcal V}{ds}=\mathcal A,\qquad
\frac{d\mathcal Q}{ds}=\mathcal V.
\end{equation}

Let
$$
Z_s=(J_s,A_s,V_s,Q_s)
$$
be the actual full state, and let
$$
\widetilde Z_s=(\widetilde J_s,\widetilde A_s,\widetilde V_s,\widetilde Q_s)
$$
be the desired state. The full state error vector is
$$
\mathcal E_s
=
Z_s-\widetilde Z_s
=
\big(\varepsilon_J(s),\varepsilon_A(s),\varepsilon_V(s),\varepsilon_Q(s)\big).
$$
Since Theorem \ref{main} is formulated for scalar equations, we do not treat the whole vector $\mathcal E_s$ directly. Instead, we choose one component of the error, for example the position error, velocity error, acceleration error, or jerk error, and regard it as a function of time and of the variables $J,\mathcal A, \mathcal V, \mathcal Q$. We denote this component error function by
$$
e=e(s,J,\mathcal A,\mathcal V,\mathcal Q).
$$
For example, if we choose the position error, then $e=e_{\mathcal Q}$ and along the actual trajectory $Z_s$ one has
$$
e_Q(s,Z_s)=\varepsilon_{\mathcal Q}(s).
$$
Thus $e_{\mathcal Q}$ is not only a number measured along one trajectory; it is a function defined for all values of $(J,\mathcal A,\mathcal V,\mathcal Q)$. The above identity only explains how this function is read along the actual trajectory. Now applying the chain rule to this error component gives 
$$
\partial_s e+J\partial_{\mathcal A}e+\mathcal A \partial_{\mathcal V}e+\mathcal V\partial_{\mathcal Q}e.
$$
This shows that the drift part is not artificially constructed; rather, it is directly generated by the jerk model.

In the classical jerk model of Mehrotra and Mahapatra \cite{control-MM}, random perturbations or unmodeled errors first enter the highest-order variable $J$. The jerk satisfies an equation of the form
$$
\frac{dJ}{ds}=-\alpha J+\omega(s),
$$
where $\omega(s)$ is a white noise input. Thus the noise occurs only in the $J$-variable, while
acceleration, velocity and position are affected through relations \eqref{chain}. In this paper, we impose a boundedness assumption on the first-order coefficient in the $J$-variable. Motivated by this structure, we use the following equation for the chosen component error function:
\begin{equation}\label{forward-ultra}
\partial_s e+J\partial_{\mathcal A}e+\mathcal A\partial_{\mathcal V}e+\mathcal V\partial_{\mathcal Q}e
-\partial_J(a(s,J)\partial_J e)+c(s,J)e+d(s,J)\partial_J e=0.
\end{equation}

{\em A natural question is: if, at the terminal time, this error function is zero for all values of the variables $J,\mathcal A,\mathcal V,\mathcal Q$ in the $L^2$ sense, can one conclude that it was already zero at all earlier times?}

\subsection{Time reversal} To align the ultraparabolic equation \eqref{forward-ultra} with our main result, Theorem \ref{main}, we first perform a time reversal. Set 
$$
t=T-s,\qquad
u(t,J,\mathcal A,\mathcal V,\mathcal Q)=e(T-t,J,\mathcal A,\mathcal V,\mathcal Q),
$$
and 
$$
\bar a(t,J)=a(T-t,J), \qquad \bar c(t, J) = c(T - t, J), \qquad \bar d(t, J) = d(T - t, J).
$$
Then \eqref{forward-ultra} is equivalent to
$$
-\partial_t u
+J\partial_{\mathcal A}u
+\mathcal A\partial_{\mathcal V}u
+\mathcal V\partial_{\mathcal Q} u
-\partial_J\big(\bar a(t,J)\partial_Ju\big)
+\bar c(t,J)u
+\bar d(t,J)\partial_Ju
=0.
$$
Multiplying the equation by $-1$, we obtain
\begin{equation}\label{backward-ultra}
\partial_t u
-J\partial_{\mathcal A}u
-\mathcal A\partial_{\mathcal V}u
-\mathcal V\partial_{\mathcal Q} u
+\partial_J\big(\bar a(t,J)\partial_Ju\big)
=
\bar c(t,J)u+\bar d(t,J)\partial_Ju.
\end{equation}
Thus the equation \eqref{forward-ultra} is written in the reversed-time form: the terminal condition $e(T,\cdot)=0$ becomes the initial condition
$u(0,\cdot)=0$.

\subsection{Application of Theorem \ref{main}.} After the above time reversal, we obtain \eqref{backward-ultra}. Now we can apply our main result, i.e., Theorem \ref{main}. Indeed, once we take $m = 1$, $n = 3$, 
$$
v=J,\qquad w=(\mathcal A,\mathcal V,\mathcal Q),
$$
and choose
$$
B_1=
\begin{pmatrix}
-1\\
0\\
0
\end{pmatrix},
\qquad
B_2=
\begin{pmatrix}
0&0&0\\
-1&0&0\\
0&-1&0
\end{pmatrix},
\qquad
A(t,v)=\bar a(t,J).
$$
Direct computation gives $\operatorname{rank} \left[ B_1, B_2 B_1, B_2^2 B_1 \right] = 3$. Then we can immediately deduce the following result by using Theorem \ref{main}:
\begin{cor}\label{cor1}
Let $u$ satisfy \eqref{backward-ultra}. Assume that $\bar a$ is uniformly elliptic and Lipschitz continuous in $(t,J)$, and $\bar c, \bar d$ are bounded. If
$$
u(0,\cdot)=0,
$$
and
$$
u, \quad \partial_J u, \quad \left( \partial_t - J \partial_{\mathcal A} - \mathcal A \partial_{\mathcal V} - \mathcal V \partial_{\mathcal Q} \right) u \in L^2((0, T) \times \mathbb{R}^4).
$$
Then
$$
u\equiv 0.
$$
\end{cor}
\begin{remark}
Corollary \ref{cor1} gives a qualitative backward uniqueness statement for \eqref{forward-ultra} satisfied by the chosen component error function. More precisely, suppose that one component of the error, such as the position error or the velocity error, is represented by a function $e=e(s,J,\mathcal A,\mathcal V,\mathcal Q)$. If this whole function is zero at the terminal time, then the same function must be identically zero for all earlier times. 
\end{remark}

 \section{Proof of the Carleman estimates}\label{proof of Carleman}
This section is devoted to the proof of the weak Carleman estimate. As described in the Introduction, we first establish Proposition \ref{local Carleman}. In order to make its proof clearer and more transparent, we separate two technical estimates from the main argument and state them as Lemmas \ref{lemma01} and \ref{lemma02}. Their proofs are postponed to Section \ref{appendix} for readability. We then derive the weak Carleman estimate, Theorem \ref{global Carleman}, by applying Proposition \ref{local Carleman} for each fixed invariant frequency and integrating over the low-frequency region $\mathcal U_R$. 

The first technical estimate in the proof of Proposition \ref{local Carleman} is stated as follows:
\begin{lemma}\label{lemma01}
Let the constants $c_0, \alpha_0$ be as in Proposition \ref{local Carleman}, $\alpha > \alpha_0$ and $t_2, b$ satisfy 
$$
0 < t_2 < \min\{ c_0 \lambda^{-2}, T \}, \qquad 0 < b < t_2.
$$
Then for any $h(t, v) \in C_c^\infty((0, t_2) \times \mathbb{R}^m)$, we have: 
\begin{align*}
&\quad \int_0^{t_2} \int_{\mathbb{R}^m} (t + b) |\Div_v \left( A \nabla_v \left( (t + b)^{-\alpha} h \right) \right)|^2 dvdt + \alpha^2 \int_0^{t_2} \int_{\mathbb{R}^m} (t + b)^{-1} |(t + b)^{-\alpha} h|^2 dvdt \\
&\quad -(2\alpha - 1) \int_0^{t_2} \int_{\mathbb{R}^m} A \nabla_v \left( (t + b)^{-\alpha} h \right) \overline{\nabla_v \left( (t + b)^{-\alpha} h \right)} dvdt \\
&\quad + \int_0^{t_2} \int_{\mathbb{R}^m} (t + b) (\partial_t A) \nabla_v \left( (t + b)^{-\alpha} h \right) \overline{\nabla_v \left( (t + b)^{-\alpha} h \right)} dvdt \\
&\ge c_1 \int_0^{t_2} \int_{\mathbb{R}^m} |\nabla_v \left( (t + b)^{-\alpha} h \right)|^2 dvdt + c_1 \alpha \int_0^{t_2} \int_{\mathbb{R}^m} (t + b)^{-1} |(t + b)^{-\alpha} h|^2 dvdt, 
\end{align*}
where the constant $c_1$ depends only on $\lambda, m$.
\end{lemma}
We next state the second technical estimate.
\begin{lemma}\label{lemma02}
Let the constants $\varepsilon_0, c_0, \alpha_0$ be as in Proposition \ref{local Carleman}, $\alpha > \alpha_0$ and $t_2, b$ satisfy 
$$
0 < t_2 < \min\{ c_0 \lambda^{-2}, T \}, \qquad 0 < b < t_2.
$$
Fix $\rho \in \mathbb{R}^n$. If 
\begin{equation*}
\sup_{0 < t < t_2} (t + b)^3 |B_1^T e^{-t B_2^T} \rho|^2 \le \varepsilon_0 \alpha.
\end{equation*}
Then, for every sufficiently small $\varepsilon>0$, there exists a constant $C_\varepsilon>0$, depending only on $\varepsilon,\lambda,m,n,B_1,B_2$, such that for every $h\in C_c^\infty((0,t_2)\times\mathbb R^m)$, one has
\begin{align*}
&\quad \left| 2Re \int_0^{t_2} \int_{\mathbb{R}^m} (t + b) i (B_1^T e^{-tB_2^T} \rho) \cdot v \left( (t + b)^{-\alpha} h \right) \overline{(\frac{\alpha}{t + b} + \Div_v \left( A \nabla_v \right)) \left( (t + b)^{-\alpha} h \right)} dvdt \right| \\
&\le \varepsilon \int_0^{t_2} \int_{\mathbb{R}^m} |\nabla_v \left( (t + b)^{-\alpha} h \right)|^2 dvdt + C_\varepsilon \varepsilon_0 \alpha \int_0^{t_2} \int_{\mathbb{R}^m} (t + b)^{-1} |(t + b)^{-\alpha} h|^2 dvdt.
\end{align*}
\end{lemma}

\bigskip
\noindent

We now prove Proposition \ref{local Carleman}.
\begin{proof}[{\bf Proof of Proposition \ref{local Carleman}}]
We divide the proof into four steps.

{\em Step 1. Decomposition without the trace term.} We first remove the trace term and work with the rest terms of $\widetilde{P}$, for simplicity, we define it as follows:
\begin{equation}\label{0rho}
\widetilde{P}_\rho^0 = \partial_t + i (B_1^T e^{-t B_2^T \rho}) \cdot v + \Div_v (A(t, v) \nabla_v).
\end{equation}
Direct computation gives that: 
\begin{align*}
(t + b)^{-\alpha} \widetilde{P}_\rho^0 h &= \partial_t \left( (t + b)^{-\alpha} h \right) + \frac{\alpha}{t + b} (t + b)^{-\alpha} h + i (B_1^T e^{-tB_2^T} \rho) \cdot v (t + b)^{-\alpha} h + \Div_v \left( A \nabla_v ((t + b)^{-\alpha} h) \right) \\
&= (K_1 + K_2) (t + b)^{-\alpha} h, 
\end{align*}
where 
\begin{equation}\label{K1}
K_1 := \partial_t + i (B_1^T e^{-tB_2^T} \rho) \cdot v
\end{equation}
is the formally skew-symmetric part and
\begin{equation}\label{K2}
K_2 := \frac{\alpha}{t + b} + \Div_v \left( A \nabla_v \right)
\end{equation}
is the formally self-adjoint part. Ignoring the trace term for the moment, we obtain 
\begin{align*}
\int_0^{t_2} \int_{\mathbb{R}^m} (t + b)^{-2\alpha + 1} |\widetilde{P}_\rho^0 h|^2 dvdt &\ge 2 Re\int_0^{t_2} \int_{\mathbb{R}^m} (t + b) K_1 \left( (t + b)^{-\alpha} h \right) \cdot \overline{K_2 \left( (t + b)^{-\alpha} h \right)} dv dt \\
&\quad + \int_0^{t_2} \int_{\mathbb{R}^m} (t + b) |K_2 \left( (t + b)^{-\alpha} h \right)|^2 dvdt.
\end{align*}
Splitting the $K_1$-term according to \eqref{K1}, we obtain
\begin{equation}\label{chaifen}
\int_0^{t_2} \int_{\mathbb{R}^m} (t + b)^{-2\alpha + 1} |\widetilde{P}_\rho^0 h|^2 dvdt \ge J_1 + J_2, 
\end{equation}
where 
\begin{equation}\label{J1}
\begin{aligned}
J_1 &= 2Re \int_0^{t_2} \int_{\mathbb{R}^m} (t + b) \partial_t \left( (t + b)^{-\alpha} h \right) \overline{K_2 \left( (t + b)^{-\alpha} h \right)} dvdt \\
&\quad + \int_0^{t_2} \int_{\mathbb{R}^m} (t + b) |K_2\left( (t + b)^{-\alpha} h \right)|^2 dvdt, 
\end{aligned}
\end{equation}
and
\begin{equation}\label{J2}
J_2 = 2Re \int_0^{t_2} \int_{\mathbb{R}^m} (t + b) i (B_1^T e^{-tB_2^T} \rho) \cdot v \left( (t + b)^{-\alpha} h \right) \overline{K_2 \left( (t + b)^{-\alpha} h \right)} dvdt.
\end{equation}

{\em Step 2. Computation of $J_1$.} We first treat the first term in $J_1$, defined in \eqref{J1}. By \eqref{K2}, we have
\begin{equation}\label{computation-J1-0}
\begin{aligned}
&\quad 2Re \int_0^{t_2} \int_{\mathbb{R}^m} (t + b) \partial_t \left( (t + b)^{-\alpha} h \right) \overline{K_2 \left( (t + b)^{-\alpha} h \right)} dvdt \\
&= 2Re \int_0^{t_2} \int_{\mathbb{R}^m} (t + b) \partial_t \left( (t + b)^{-\alpha} h \right) \overline{\Div_v \left( A \nabla_v \left( (t + b)^{-\alpha} h \right) \right)} dvdt, 
\end{aligned}
\end{equation}
where we use the fact that 
$$
2\alpha Re\int_0^{t_2} \int_{\mathbb{R}^m} \partial_t \left( (t + b)^{-\alpha} h \right) \cdot \overline{(t + b)^{-\alpha} h} dvdt = 0, 
$$
because $(t+b)^{-\alpha}h$ is compactly supported in $(0, t_2)$. Then by integration by parts, we have 
\begin{equation}\label{computation-J1-1}
\begin{aligned}
&\quad 2Re \int_0^{t_2} \int_{\mathbb{R}^m} (t + b) \partial_t \left( (t + b)^{-\alpha} h \right) \overline{\Div_v \left( A \nabla_v \left( (t + b)^{-\alpha} h \right) \right)} dvdt \\
&= \int_0^{t_2} \int_{\mathbb{R}^m} A \nabla_v \left( (t + b)^{-\alpha} h \right) \overline{\nabla_v \left( (t + b)^{-\alpha} h \right)} dvdt \\
&\quad + \int_0^{t_2} \int_{\mathbb{R}^m} (t + b) (\partial_t A) \nabla_v \left( (t + b)^{-\alpha} h \right) \overline{\nabla_v \left( (t + b)^{-\alpha} h \right)} dvdt.
\end{aligned}
\end{equation}
We next compute the second term in $J_1$. Using \eqref{K2}, we get
\begin{equation}\label{computation-J1-2}
\begin{aligned}
\int_0^{t_2} \int_{\mathbb{R}^m} (t + b) |K_2 \left( (t + b)^{-\alpha} h \right)|^2 dvdt &= \alpha^2 \int_0^{t_2} \int_{\mathbb{R}^m} (t + b)^{-1} |(t + b)^{-\alpha} h|^2 dvdt \\
&\quad + \int_0^{t_2} \int_{\mathbb{R}^m} (t + b) |\Div_v \left( A \nabla_v \left( (t + b)^{-\alpha} h  \right) \right)|^2 dvdt \\
&\quad - 2\alpha \int_0^{t_2} \int_{\mathbb{R}^m} A \nabla_v \left( (t + b)^{-\alpha} h \right) \overline{\nabla_v \left( (t + b)^{-\alpha} h \right)} dvdt.
\end{aligned}
\end{equation}
Combining \eqref{computation-J1-0} -- \eqref{computation-J1-2}, we obtain  
\begin{equation}\label{exp-J1}
\begin{aligned}
J_1 &= \int_0^{t_2} \int_{\mathbb{R}^m} (t + b) |\Div_v \left( A \nabla_v \left( (t + b)^{-\alpha} h \right) \right)|^2 dvdt + \alpha^2 \int_0^{t_2} \int_{\mathbb{R}^m} (t + b)^{-1} |(t + b)^{-\alpha} h|^2 dvdt \\
&\quad -(2\alpha - 1) \int_0^{t_2} \int_{\mathbb{R}^m} A \nabla_v \left( (t + b)^{-\alpha} h \right) \overline{\nabla_v \left( (t + b)^{-\alpha} h \right)} dvdt \\
&\quad + \int_0^{t_2} \int_{\mathbb{R}^m} (t + b) (\partial_t A) \nabla_v \left( (t + b)^{-\alpha} h \right) \overline{\nabla_v \left( (t + b)^{-\alpha} h \right)} dvdt.
\end{aligned}
\end{equation}

{\em Step 3. Lower bound without the trace term.} 
By Lemma \ref{lemma01}, we have 
$$
J_1 \ge c_1 \int_0^{t_2} \int_{\mathbb{R}^m} |\nabla_v \left( (t + b)^{-\alpha} h \right)|^2 dvdt + c_1 \alpha \int_0^{t_2} \int_{\mathbb{R}^m} (t + b)^{-1} |(t + b)^{-\alpha} h|^2 dvdt.
$$
Lemma \ref{lemma02} gives 
$$
|J_2| \le \varepsilon \int_0^{t_2} \int_{\mathbb{R}^m} |\nabla_v \left( (t + b)^{-\alpha} h \right)|^2 dvdt + C_\varepsilon \varepsilon_0 \alpha \int_0^{t_2} \int_{\mathbb{R}^m} (t + b)^{-1} |(t + b)^{-\alpha} h|^2 dvdt.
$$
We now choose $\varepsilon > 0$ small enough that $\varepsilon \le \frac{c_1}{2}$, then choose $\varepsilon_0$ small enough that $C_\varepsilon \varepsilon_0 \le \frac{c_1}{2}$. Thus, we have
\begin{equation}\label{lower-bound-tmp}
\int_0^{t_2} \int_{\mathbb{R}^m} (t + b)^{-2\alpha + 1} |\widetilde{P}_\rho^0 h|^2 dvdt \ge C_1 \int_0^{t_2} \int_{\mathbb{R}^m} (t + b)^{-2\alpha} |\nabla_v h|^2 dvdt + C_1 \alpha \int_0^{t_2} \int_{\mathbb{R}^m} (t + b)^{-2\alpha - 1} |h|^2 dvdt.
\end{equation} 

{\em Step 4. Absorption of the trace term.} Recall definitions \eqref{def-01} and \eqref{0rho}, we have 
$$
|\widetilde{P}_\rho^0 h|^2 \le 2 |\widetilde{P}_\rho h|^2 + \frac12 |\operatorname{tr} B_2|^2 |h|^2.
$$
Thus \eqref{lower-bound-tmp} gives:
\begin{align*}
&\quad \int_0^{t_2} \int_{\mathbb{R}^m} (t + b)^{-2\alpha} |\nabla_v h|^2 dvdt + \alpha \int_0^{t_2} \int_{\mathbb{R}^m} (t + b)^{-2\alpha - 1} |h|^2 dvdt \\
&\le C_2 \int_0^{t_2} \int_{\mathbb{R}^m} (t + b)^{-2\alpha + 1} |\widetilde{P}_\rho h|^2 dvdt + C_2 |\operatorname{tr} B_2|^2 \int_0^{t_2} \int_{\mathbb{R}^m} (t + b)^{-2\alpha + 1} |h|^2 dvdt.
\end{align*}
Since 
$$
(t + b)^{-2\alpha + 1} = (t + b)^2 (t + b)^{-2\alpha - 1} 
$$
and $0 < b < t_2$, then 
$$
C_2 |\operatorname{tr} B_2|^2 \int_0^{t_2} \int_{\mathbb{R}^m} (t + b)^{-2\alpha + 1} |h|^2 dvdt \le 4C_2 |\operatorname{tr} B_2|^2 t_2^2 \int_0^{t_2} \int_{\mathbb{R}^m} (t + b)^{-2\alpha - 1} |h|^2 dvdt.
$$
Considering $t_2\leq c_0\lambda^{-2}$, choosing $\alpha_0$ sufficiently large, depending on $\lambda$ and $B_2$, allows the last term to be absorbed into 
$$
\alpha \int_0^{t_2} \int_{\mathbb{R}^m} (t + b)^{-2\alpha - 1} |h|^2 dvdt.
$$
This proves \eqref{Carleman-2} and completes the proof of Proposition \ref{local Carleman}.
\end{proof}

\bigskip
We next derive Theorem \ref{global Carleman} from Proposition \ref{local Carleman}. The idea is to apply Proposition \ref{local Carleman} for each fixed invariant frequency $\rho$, and then integrate over the low-frequency region $\mathcal U_R$.
\begin{proof}[{\bf Proof of Theorem \ref{global Carleman}}]
By definitions \eqref{T} and \eqref{S}, 
$$
\mathcal{S}_R g = \mathcal{T}^{-1} \mathbf{1}_{\mathcal{U}_R} \mathcal{T} g,  
$$ 
we have 
$$
\mathcal{T} (\mathcal{S}_R g) (t, v, \rho) = \mathbf{1}_{\mathcal{U}_R} (\rho) (\mathcal{T} g) (t, v, \rho).
$$
Since $\mathcal{T}$ occurs only on the degenerate variables, it commutes with $\nabla_v$. Hence 
$$
\mathcal{T} (\nabla_v \mathcal{S}_R g) = \nabla_v \mathcal{T} (\mathcal{S}_R g) = \mathbf{1}_{\mathcal{U} _R} (\rho) \nabla_v (\mathcal{T} g).
$$
By Plancherel's theorem, we have 
\begin{equation}\label{prop-01}
\int_0^{t_2} \int_{\mathbb{R}^m} \int_{\mathbb{R}^n} (t + b)^{-2\alpha} |\nabla_v \mathcal{S}_R g|^2 dvdwdt = \int_{\mathcal{U}_R} \int_0^{t_2} \int_{\mathbb{R}^m} (t + b)^{-2\alpha} |\nabla_v (\mathcal{T} g)|^2 dvdtd\rho,
\end{equation}
and similarly, 
\begin{equation}\label{prop-02}
\int_0^{t_2} \int_{\mathbb{R}^m} \int_{\mathbb{R}^n} (t + b)^{-2\alpha - 1} |\mathcal{S}_R g|^2 dvdwdt = \int_{\mathcal{U}_R} \int_0^{t_2} \int_{\mathbb{R}^m} (t + b)^{-2\alpha - 1} |\mathcal{T} g|^2 dvdtd\rho.
\end{equation}

For every fixed $\rho \in \mathcal{U}_R$, set 
$$
h(t, v) = (\mathcal{T} g) (t, v, \rho) \in C_c^\infty ((0, t_2) \times \mathbb{R}^m).
$$
Choose $C_* \ge \varepsilon_0^{-1}$. Then the assumption $\alpha \ge \alpha_0 + C_* R$ implies $R \le \varepsilon_0 \alpha$. Therefore, for every $\rho \in \mathcal U_R$, 
$$
\sup_{0 < t < t_2} (t + b)^3 |B_1^T e^{-tB_2^T} \rho|^2 \le \varepsilon_0 \alpha.
$$
Thus we can apply Proposition \ref{local Carleman} to $h(t, v) = (\mathcal{T} g) (t, v, \rho)$, for every $\rho \in \mathcal{U}_R$, we have 
\begin{align*}
&\quad \int_0^{t_2} \int_{\mathbb{R}^m} (t + b)^{-2\alpha} |\nabla_v (\mathcal{T} g)|^2 dvdt + \alpha \int_0^{t_2} \int_{\mathbb{R}^m} (t + b)^{-2\alpha - 1} |\mathcal{T} g|^2 dvdt \\
&\le C\int_0^{t_2} \int_{\mathbb{R}^m} (t + b)^{-2\alpha + 1} |\widetilde{P}_\rho (\mathcal{T} g)|^2 dvdt.
\end{align*}
Integrating above inequality with respect to $\rho \in \mathcal{U}_R$, we have 
\begin{equation}\label{propp}
\begin{aligned}
&\quad \int_{\mathcal{U}_R} \int_0^{t_2} \int_{\mathbb{R}^m} (t + b)^{-2\alpha} |\nabla_v (\mathcal{T} g)|^2 dvdt d\rho + \alpha \int_{\mathcal{U}_R} \int_0^{t_2} \int_{\mathbb{R}^m} (t + b)^{-2\alpha - 1} |\mathcal{T} g|^2 dvdtd\rho \\
&\le C \int_{\mathcal{U}_R} \int_0^{t_2} \int_{\mathbb{R}^m} (t + b)^{-2\alpha + 1} |\widetilde{P}_\rho (\mathcal{T} g)|^2 dvdtd\rho.
\end{aligned}
\end{equation}
Moreover, using $\mathcal T P \mathcal T^{-1}=\widetilde P_\rho$ and
$\mathcal T(\mathcal S_Rg)={\bf 1}_{\mathcal U_R} \mathcal Tg$, we have
$$
\mathcal T(P \mathcal S_Rg) = \widetilde P_\rho({\bf 1}_{\mathcal U_R}\mathcal Tg).
$$
Since $\widetilde P_\rho$ contains no derivatives with respect to $\rho$, and
${\bf 1}_{\mathcal U_R}$ depends only on $\rho$, we get
$$
 \widetilde P_\rho({\bf 1}_{\mathcal U_R}\mathcal Tg) = {\bf 1}_{\mathcal U_R}\widetilde P_\rho(\mathcal Tg).
$$
Therefore
$$
\mathcal T(P \mathcal S_Rg) = {\bf 1}_{\mathcal U_R}\widetilde P_\rho(\mathcal Tg).
$$
Therefore, by Plancherel's Theorem, we have 
\begin{equation}\label{propo}
\int_{\mathcal{U}_R} \int_0^{t_2} \int_{\mathbb{R}^m}  (t + b)^{-2\alpha + 1} |\widetilde{P}_\rho (\mathcal{T} g)|^2 dvdtd\rho = \int_0^{t_2} \int_{\mathbb{R}^m} \int_{\mathbb{R}^n} (t + b)^{-2\alpha + 1} |P \mathcal{S}_R g|^2 dvdwdt.
\end{equation}
Combining the results \eqref{prop-01}, \eqref{prop-02}, \eqref{propo} with \eqref{propp}, we obtain the desired estimate.
\end{proof}

\section{Proof of the main result}\label{proof of BU}
In this section, we use the weak Carleman estimate, Theorem \ref{global Carleman}, to prove a local-in-time backward uniqueness result. We then apply this local result iteratively to prove Theorem \ref{main}.

We first record a consequence of the rank condition \eqref{condition-B1B2} which will be used to pass from
$\mathcal S_Ru=0$ to $u=0$.
\begin{lemma}\label{condition-B1B2-proof}
Assume the rank condition \eqref{condition-B1B2} holds. Let $c_0$ be the constant in Theorem \ref{global Carleman}, and let 
$$
0<t_2\leq \min\{c_0\lambda^{-2},T\},\qquad 0<b\leq t_2.
$$
Then there exists a constant $c_2 > 0$ depending only on $t_2, B_1, B_2$, such that for any $\rho\in\mathbb R^n$, 
$$
\sup_{0<t<t_2}(t+b)^3 \left|B_1^Te^{-tB_2^T}\rho\right|^2 \geq c_2 |\rho|^2.
$$
Consequently, for $\mathcal U_R$ defined by \eqref{low}, the set 
$\mathcal U_R$ is bounded for every $R>0$, and 
$$
\bigcup_{R>0}\mathcal U_R=\mathbb R^n.
$$
\end{lemma}
\begin{proof}
For simplicity, we set, for $\rho\in\mathbb R^n$, 
$$
G(\rho)=\max_{0\leq t\leq t_2/2}
t^3\left|B_1^Te^{-tB_2^T}\rho\right|^2.
$$
We claim that $G(\rho) > 0$ for any $\rho \ne 0$. Indeed, if $G(\rho) = 0$, then we have 
$$
B_1^Te^{-tB_2^T}\rho=0,\quad 0<t<t_2/2.
$$
Differentiating above equality at $t = 0$ up to order $n - 1$, we obtain 
$$
B_1^T(B_2^T)^k\rho=0,\qquad k = 0, 1, \dots, n - 1.
$$
Thus 
$$
\left[ B_1,B_2B_1,\ldots,B_2^{n-1}B_1 \right]^T \rho = 0.
$$
However, by the rank condition \eqref{condition-B1B2}, the matrix $\left[ B_1,B_2B_1,\ldots,B_2^{n-1}B_1 \right]$ has rank $n$, which means that $\rho = 0$.

As the analysis above, we have $G(\rho) > 0$ for any $\rho \in \mathbb{R}^n$ with $|\rho| = 1$. By using the continuous and homogeneity of $G(\rho)$, there exists a constant $c_2$ depending only on $t_2, B_1, B_2$, such that 
$$
G(\rho)\geq c_2|\rho|^2,\quad \rho\in\mathbb R^n.
$$
Since $(t+b)^3\geq t^3$, we have 
$$
\sup_{0<t<t_2}(t+b)^3 \left|B_1^Te^{-tB_2^T}\rho\right|^2
\geq G(\rho)\geq c_2|\rho|^2.
$$

Therefore, if $\rho\in\mathcal U_R$, then $c_2|\rho|^2\leq R$, and hence $\mathcal U_R$ is bounded. For fixed $\rho \in \mathbb{R}^n$, since
$$
\sup_{0<t<t_2}(t+b)^3 \left|B_1^Te^{-tB_2^T}\rho\right|^2<\infty, 
$$
then $\rho\in\mathcal U_R$ for all sufficiently large $R > 0$, this proves $\bigcup_{R>0}\mathcal U_R=\mathbb R^n$.
\end{proof}

We next prove the following local-in-time backward uniqueness result.
\begin{prop}\label{local time interval}
Assume that the hypotheses of Theorem \ref{main} hold. Let $c_0 > 0$ be the constant in Theorem \ref{global Carleman}. Let
$$
0 < t_1 < t_2 < \min \{ c_0 \lambda^{-2}, T \}.
$$
Then 
$$
u \equiv 0, \quad a.e. \ (t, v, w) \in (0, t_1) \times \mathbb{R}^m \times \mathbb{R}^n.
$$
\end{prop}
\begin{proof}
Fix $R > 0$. By $\mathcal S_R=\mathcal T^{-1}{\bf 1}_{\mathcal U_R} \mathcal T$ and $\mathcal T P \mathcal T^{-1}=\widetilde P_\rho$, we have
$$
\mathcal T(P \mathcal S_Ru) = \widetilde P_\rho({\bf 1}_{\mathcal U_R}\mathcal Tu).
$$
Since $\widetilde P_\rho$ contains no derivatives with respect to $\rho$, and ${\bf 1}_{\mathcal U_R}$ depends only on $\rho$, it follows that
$$
\widetilde P_\rho({\bf 1}_{\mathcal U_R}\mathcal Tu) = {\bf 1}_{\mathcal U_R}\widetilde P_\rho(\mathcal Tu).
$$
Therefore
$$
\mathcal T(P \mathcal S_Ru) = {\bf 1}_{\mathcal U_R}\widetilde P_\rho(\mathcal Tu) = \mathcal T(\mathcal S_RPu),
$$
and hence
$$
P \mathcal S_Ru=\mathcal S_R Pu.
$$
Similarly, since $T$ does not change the $v$-variable and
${\bf 1}_{\mathcal U_R}$ is independent of $v$, we have
$$
\nabla_v \mathcal S_Ru=\mathcal S_R\nabla_vu.
$$
Moreover, since $c(t,v)$ and $d(t,v)$ are independent of $w$,
$$
\mathcal S_R(cu)=c\mathcal S_Ru,\qquad \mathcal S_R(d\cdot\nabla_vu)=d\cdot\nabla_v \mathcal S_Ru.
$$
Using $Pu=cu+d\cdot\nabla_vu$, we obtain
$$
P(\mathcal S_Ru) = \mathcal S_RPu = \mathcal S_R(cu+d\cdot\nabla_vu) = c\mathcal S_Ru+d\cdot\nabla_v \mathcal S_Ru.
$$

Choose a smooth cutoff function $\chi(t) \in [0, 1]$ satisfies:  
\begin{equation}\label{chi}
\chi(t) = 
\begin{cases}
1, \quad t \le t_1, \\
0, \quad t > t_2.
\end{cases}
\end{equation}
Then $supp \chi'(t) \subset (t_1, t_2)$, and $\chi(t) \mathcal{S}_R u \big|_{t = 0} = 0$. Since $\chi(t)$ is independent of $w$, we also have $\mathcal S_R(\chi u)=\chi \mathcal S_Ru$. Thus 
$$
P \mathcal S_R(\chi u) = P(\chi \mathcal S_Ru) = \chi P(\mathcal S_Ru)+\chi' \mathcal S_Ru,
$$
and hence
$$
P \mathcal S_R(\chi u) = c\chi \mathcal S_Ru+d\cdot\nabla_v(\chi \mathcal S_Ru)+\chi' \mathcal S_Ru .
$$
Although Theorem \ref{global Carleman} stated for functions in $C_c^\infty((0,t_2)\times\mathbb R^m\times\mathbb R^n)$, estimate \eqref{Carleman-1} extends by density to all functions $g$ such that
$$
\mathcal S_Rg,\quad \nabla_v \mathcal S_Rg,\quad P \mathcal S_Rg
\in L^2((0,t_2)\times\mathbb R^{m} \times \mathbb{R}^n), 
$$
provided that $\mathcal S_Rg$ has compact support in time in $[0,t_2)$ and has zero $L^2$-trace at $t = 0$. Indeed, this follows by applying \eqref{Carleman-1} to smooth compactly supported approximations $g_k$ satisfying:
$$
\mathcal S_Rg_k\to \mathcal S_Rg,\qquad
\nabla_v \mathcal S_Rg_k\to\nabla_v \mathcal S_Rg,\qquad
P \mathcal S_Rg_k\to P \mathcal S_Rg, 
$$
in $L^2$, and then passing to the limit. We apply this extension with $g=\chi u$. Since
$$
P \mathcal S_R(\chi u)=P(\chi \mathcal S_Ru) = c\chi \mathcal S_Ru+d\cdot\nabla_v(\chi \mathcal S_Ru)+\chi' \mathcal S_Ru \in L^2, 
$$
hence Theorem \ref{global Carleman} applies to $g = \chi u$. For $\alpha \ge \alpha_0 + C_*R$, we obtain  
\begin{equation}\label{apply Carleman}
\begin{aligned}
&\quad \int_0^{t_2} \int_{\mathbb{R}^m} \int_{\mathbb{R}^n} (t + b)^{-2\alpha} |\nabla_v (\chi \mathcal{S}_R u)|^2 dwdvdt + \alpha \int_0^{t_2} \int_{\mathbb{R}^m} \int_{\mathbb{R}^n} (t + b)^{-2\alpha - 1} |\chi \mathcal{S}_R u|^2 dwdvdt \\
&\le 3C \int_0^{t_2} \int_{\mathbb{R}^m} \int_{\mathbb{R}^n} (t + b)^{-2\alpha + 1} |c\chi \mathcal{S}_R u|^2 dwdvdt + 3C \int_0^{t_2} \int_{\mathbb{R}^m} \int_{\mathbb{R}^n} (t + b)^{-2\alpha + 1} |d \cdot \nabla_v (\chi \mathcal{S}_R u)|^2 dwdvdt\\
&\quad + 3C \int_0^{t_2} \int_{\mathbb{R}^m} \int_{\mathbb{R}^n} (t + b)^{-2\alpha + 1} |\chi' \mathcal{S}_R u|^2 dwdvdt \\
&= I_1 + I_2 + I_3.
\end{aligned}
\end{equation}

For $I_1$, since 
$$
(t + b)^{-2\alpha + 1} = (t + b)^2 (t + b)^{-2\alpha - 1} \le 4t_2^2 (t + b)^{-2\alpha - 1}, 
$$
and $|c| \le \lambda$ by \eqref{condition-3}, we have
$$
I_1 \le 12C \lambda^2 t_2^2 \int_0^{t_2} \int_{\mathbb{R}^m} \int_{\mathbb{R}^n} (t + b)^{-2\alpha - 1} |\chi \mathcal{S}_R u|^2 dwdvdt.
$$ 
Choosing $c_0$ sufficiently small so that $12C c_0^2 \le \frac14$ and using $\alpha \ge \alpha_0$, the term $I_1$ can be absorbed by 
$$
\alpha \int_0^{t_2} \int_{\mathbb{R}^m} \int_{\mathbb{R}^n} (t + b)^{-2\alpha - 1} |\chi \mathcal{S}_R u|^2 dwdvdt.
$$

For $I_2$, since $t + b \le 2t_2$, we have 
$$
I_2 \le 6C t_2 \lambda^2 \int_0^{t_2} \int_{\mathbb{R}^m} \int_{\mathbb{R}^n} (t + b)^{-2\alpha} |\nabla_v (\chi \mathcal{S}_R u)|^2 dwdvdt.
$$
Since $t_2 \le c_0 \lambda^{-2}$, choosing $c_0$ sufficiently small so that 
$
6Cc_0 \leq \frac14
$,
then $I_2$ can be absorbed by 
$$
\int_0^{t_2} \int_{\mathbb{R}^m} \int_{\mathbb{R}^n} (t + b)^{-2\alpha} |\nabla_v (\chi \mathcal{S}_R u)|^2 dwdvdt.
$$

For $I_3$, since $supp \chi'(t) \subset (t_1, t_2)$, we have 
$$
I_3 \le C_\chi \int_{t_1}^{t_2} \int_{\mathbb{R}^m} \int_{\mathbb{R}^n} (t + b)^{-2\alpha + 1} |\mathcal{S}_R u|^2 dwdvdt, 
$$
where $C_\chi$ depends on the chosen $\chi$, but is independent of $R, \alpha$ and $u$.

Combining \eqref{apply Carleman} with the estimates for $I_1$, $I_2$ and $I_3$, we obtain  
\begin{align*}
&\quad \int_0^{t_2} \int_{\mathbb{R}^m} \int_{\mathbb{R}^n} (t + b)^{-2\alpha} |\nabla_v (\chi \mathcal{S}_R u)|^2 dwdvdt + \alpha \int_0^{t_2} \int_{\mathbb{R}^m} \int_{\mathbb{R}^n} (t + b)^{-2\alpha - 1} |\chi \mathcal{S}_R u|^2 dwdvdt \\
&\le C_\chi \int_{t_1}^{t_2} \int_{\mathbb{R}^m} \int_{\mathbb{R}^n} (t + b)^{-2\alpha + 1} |\mathcal{S}_R u|^2 dwdvdt. 
\end{align*}
In particular,  
$$
\alpha \int_0^{t_2} \int_{\mathbb{R}^m} \int_{\mathbb{R}^n} (t + b)^{-2\alpha - 1} |\chi \mathcal{S}_R u|^2 dwdvdt \le C_\chi \int_{t_1}^{t_2} \int_{\mathbb{R}^m} \int_{\mathbb{R}^n} (t + b)^{-2\alpha + 1} |\mathcal{S}_R u|^2 dwdvdt.
$$
Since $\chi = 1$ on $0 < t < t_1$, the above inequality implies 
\begin{equation}\label{555}
\alpha \int_0^{t_1} \int_{\mathbb{R}^m} \int_{\mathbb{R}^n} (t + b)^{-2\alpha - 1} |\mathcal{S}_R u|^2 dwdvdt \le C_\chi \int_{t_1}^{t_2} \int_{\mathbb{R}^m} \int_{\mathbb{R}^n} (t + b)^{-2\alpha + 1} |\mathcal{S}_R u|^2 dwdvdt.
\end{equation}

Since 
$$
\alpha \int_0^{t_1} \int_{\mathbb{R}^m} \int_{\mathbb{R}^n} (t + b)^{-2\alpha - 1} |\mathcal{S}_R u|^2 dwdvdt \ge \alpha (b + t_1)^{-2\alpha - 1} \int_0^{t_1} \int_{\mathbb{R}^m} \int_{\mathbb{R}^n} |\mathcal{S}_R u |^2 dwdvdt, 
$$
and 
$$
C_\chi \int_{t_1}^{t_2} \int_{\mathbb{R}^m} \int_{\mathbb{R}^n} (t + b)^{-2\alpha + 1} |\mathcal{S}_R u|^2 dwdvdt \le C_\chi (b + t_1)^{-2\alpha + 1} \int_{t_1}^{t_2} \int_{\mathbb{R}^m} \int_{\mathbb{R}^n} |\mathcal{S}_R u|^2 dwdvdt, 
$$
hence \eqref{555} gives: 
$$
\int_0^{t_1} \int_{\mathbb{R}^m} \int_{\mathbb{R}^n} |\mathcal{S}_R u |^2 dwdvdt \le \frac{C_\chi (b + t_1)^2}{\alpha} \int_{t_1}^{t_2} \int_{\mathbb{R}^m} \int_{\mathbb{R}^n} |\mathcal{S}_R u|^2 dwdvdt.
$$
For fixed $R>0$, the unitarity of $T$ gives
$$
        \|\mathcal S_Ru\|_{L^2((t_1,t_2)\times\mathbb R^{m} \times \mathbb{R}^n)} = \|{\bf 1}_{\mathcal U_R}\mathcal Tu\|_{L^2((t_1,t_2)\times\mathbb R^m\times\mathbb R^n)} \leq \|\mathcal Tu\|_{L^2((t_1,t_2)\times\mathbb R^m\times\mathbb R^n)} = \|u\|_{L^2((t_1,t_2)\times\mathbb R^{m} \times \mathbb R^n)}.
$$
Hence the integral on the right-hand side is finite and independent of $\alpha$. Letting $\alpha\to\infty$, we obtain:
\begin{equation}\label{990}
\mathcal{S}_R u \equiv 0, \quad a.e.\ (t, v, w) \in (0, t_1) \times \mathbb{R}^m \times \mathbb{R}^n.
\end{equation}
By Lemma \ref{condition-B1B2-proof}, $ \bigcup_{R>0}\mathcal U_R=\mathbb R^n$. Since the sets $\mathcal U_R$ are increasing with respect to $R$, $\mathbf 1_{\mathcal U_R}\to1$ pointwise. As $\mathcal T$ is unitary and $\mathcal S_R=\mathcal T^{-1}\mathbf 1_{\mathcal U_R} \mathcal T$, the dominated convergence theorem gives $\mathcal S_Ru\to u$ in $L^2((0,t_1)\times\mathbb R^m\times\mathbb R^n)$ as $R\to\infty$. Letting $R\to\infty$ in \eqref{990}, in the $L^2$-sense:
$$
u = 0, \quad a.e.\ (t, v, w) \in (0, t_1) \times \mathbb{R}^m \times \mathbb{R}^n.
$$
\end{proof}

\bigskip
\noindent
We now prove Theorem \ref{main} by iterating Proposition \ref{local time interval}.
\begin{proof}[{\bf Proof of Theorem \ref{main}.}]
Let $T' \in (0,T)$ be arbitrary, and set 
$$
\delta = \frac14 \min\{ c_0 \lambda^{-2}, T', T - T' \}.
$$
Taking $t_1 = \delta$ and $t_2 = 2\delta$, Proposition \ref{local time interval} gives 
$$
u = 0, \quad a.e.\ (t, v, w) \in (0, \delta) \times \mathbb{R}^m \times \mathbb{R}^n.
$$

We now show that $u = 0$ on later time interval. Suppose that $u = 0$ a.e. on $(0, s) \times\mathbb R^{m} \times \mathbb{R}^n$ for some $s<T'$. Choose $s_0 \in (s - \frac{\delta}{2}, s)$ such that $u(s_0) = 0$ in $L^2$ sense. Such an $s_0$ exists because $u=0$ for a.e. $t\in(s-\delta/2,s)$. Define
$$
U(\tau,v,w)=u(s_0+\tau,v,w),\quad 0<\tau<T-s_0.
$$
The shifted coefficients
$$
A_{s_0}(\tau,v)=A(s_0+\tau,v),\qquad
c_{s_0}(\tau,v)=c(s_0+\tau,v),\qquad
d_{s_0}(\tau,v)=d(s_0+\tau,v), 
$$
satisfy the same assumptions as $A,c,d$, and $U(0,\cdot,\cdot)=0$. Since
$$
2\delta\leq c_0\lambda^{-2},\quad s_0+2\delta<T,
$$
Proposition \ref{local time interval} applied to $U$ gives
$$
U = 0, \quad a.e.\ (\tau, v, w) \in (0,\delta)\times\mathbb R^m\times\mathbb R^n.
$$
Thus 
$$
u = 0, \quad a.e. \ (t, v, w) \in (s_0,s_0+\delta)\times\mathbb R^m\times\mathbb R^n.
$$
Since $s_0 > s - \frac{\delta}{2}$, we have 
$$
u = 0, \quad a.e.\ (t, v, w) \in (0,s+\frac{\delta}{2})\times\mathbb R^m\times\mathbb R^n.
$$

Repeating the above argument finitely many times, we obtain:
$$
u=0, \quad a.e.\ (t, v, w) \in (0,T')\times\mathbb R^m\times\mathbb R^n.
$$
Since $T'\in(0,T)$ is arbitrary, we conclude that
$$
u=0, \quad a.e.\ (t, v, w) \in (0,T)\times\mathbb R^m\times\mathbb R^n.
$$
\end{proof}

\section{Proofs of the Technical Lemmas}\label{appendix}
In this section, we prove Lemmas \ref{lemma01} and \ref{lemma02}. 
\begin{proof}[{\bf Proof of Lemma \ref{lemma01}}]
For simplicity, we fix $t \in (0, t_2)$ and focus on the spatial variables. Define 
$$
E(t) = \int_{\mathbb{R}^m} A \nabla_v \left( (t + b)^{-\alpha} h \right) \overline{\nabla_v \left( (t + b)^{-\alpha} h \right)} dv.
$$
Define $Q(t)$ to be the part of the left-hand side of Lemma \ref{lemma01} which does not contain the $\partial_t A$ term:\begin{align*}
Q(t) &= (t + b) \int_{\mathbb{R}^m} |\Div_v \left( A \nabla_v \left( (t + b)^{-\alpha} h \right) \right)|^2 dv + \alpha^2 (t + b)^{-1} \int_{\mathbb{R}^m} |(t + b)^{-\alpha} h|^2 dv - (2\alpha - 1)E(t).
\end{align*}
Thus we have 
\begin{equation}\label{lemma01-1}
Q(t) = \left\| (t + b)^{1/2} \Div_v \left( A \nabla_v \left( (t + b)^{-\alpha} h \right) \right) + \frac{\alpha}{(t + b)^{1/2}} (t + b)^{-\alpha} h \right\|^2_{L_v^2} + E(t), 
\end{equation}
where we use the fact that 
\begin{equation}\label{lemma01-tmp}
E(t) = \int_{\mathbb{R}^m} A \nabla_v \left( (t + b)^{-\alpha} h \right) \overline{\nabla_v \left( (t + b)^{-\alpha} h \right)} dv = -Re \int_{\mathbb{R}^m} \Div_v \left( A \nabla_v \left( (t + b)^{-\alpha} h \right) \right) \overline{(t + b)^{-\alpha} h} dv.
\end{equation}
From \eqref{lemma01-1}, we have 
\begin{equation}\label{lemma01-2}
Q(t) \ge E(t).
\end{equation}
On the other hand, from \eqref{lemma01-tmp},  
$$
E(t) \le \frac{(t + b)^{1/2} \left\| \Div_v \left( A \nabla_v \left( (t + b)^{-\alpha} h \right) \right) \right\|_{L_v^2} \cdot \frac{\alpha}{(t + b)^{1/2}} \left\| (t + b)^{-\alpha} h \right\|_{L_v^2}}{\alpha}, 
$$
then combining it with \eqref{lemma01-1}, we have
\begin{align*}
Q(t) &\ge \left( (t + b)^{1/2} \left\| \Div_v \left( A \nabla_v \left( (t + b)^{-\alpha} h \right) \right) \right\|_{L_v^2} \right)^2 + \left( \frac{\alpha}{(t + b)^{1/2}} \left\| (t + b)^{-\alpha} h \right\|_{L_v^2} \right)^2 \\
&\quad - (2 - \frac{1}{\alpha}) \left( (t + b)^{1/2} \left\| \Div_v \left( A \nabla_v \left( (t + b)^{-\alpha} h \right) \right) \right\|_{L_v^2} \right) \cdot \left( \frac{\alpha}{(t + b)^{1/2}} \left\| (t + b)^{-\alpha} h \right\|_{L_v^2} \right).
\end{align*}
By using the condition $\alpha \ge \alpha_0 > 1$, treating the right-hand side of above inequality as a quadratic function and completing the square, we have 
\begin{equation}\label{lemma01-5}
Q(t) \ge \frac34 \alpha (t + b)^{-1} \int_{\mathbb{R}^m} |(t + b)^{-\alpha} h|^2 dv.
\end{equation}
Combining \eqref{lemma01-2} and \eqref{lemma01-5} and integrating in time, we obtain: 
\begin{equation}\label{lemma01-6}
\begin{aligned}
&\quad \int_0^{t_2} \int_{\mathbb{R}^m} (t + b) |\Div_v \left( A \nabla_v \left( (t + b)^{-\alpha} h \right) \right)|^2 dvdt + \alpha^2 \int_0^{t_2} \int_{\mathbb{R}^m} (t + b)^{-1} |(t + b)^{-\alpha} h|^2 dvdt \\
&\quad - (2\alpha - 1) \int_0^{t_2} \int_{\mathbb{R}^m} A \nabla \left( (t + b)^{-\alpha} h \right) \overline{\nabla \left( (t + b)^{-\alpha} h \right)} dvdt \\
&\ge \frac38 \int_0^{t_2} \int_{\mathbb{R}^m} A \nabla \left( (t + b)^{-\alpha} h \right) \overline{\nabla \left( (t + b)^{-\alpha} h \right)} dvdt + \frac38 \alpha \int_0^{t_2} \int_{\mathbb{R}^m} (t + b)^{-1} |(t + b)^{-\alpha} h|^2 dvdt.
\end{aligned}
\end{equation}
It remains to control $(\partial_t A)$ term. By \eqref{condition-A} and \eqref{condition-3}, there exists a constant $C_m > 0$ depending only on $m$ such that 
\begin{align*}
|(\partial_t A) \nabla_v \left( (t + b)^{-\alpha} h \right) \overline{\nabla_v \left( (t + b)^{-\alpha} h \right)}| &\le C_m \lambda |\nabla_v \left( (t + b)^{-\alpha} h \right)|^2 \\
&\le C_m \lambda^2 A \nabla_v \left( (t + b)^{-\alpha} h \right) \overline{\nabla_v \left( (t + b)^{-\alpha} h \right)}.
\end{align*}
Meanwhile, since $0 < b \le t_2 \le c_0 \lambda^{-2}$, combining above inequality, $(\partial_t A)$ term has the following estimate: 
\begin{equation}\label{lemma01-8}
\begin{aligned}
&\quad \left|\int_0^{t_2} \int_{\mathbb{R}^m} (t + b) (\partial_t A) \nabla_v \left( (t + b)^{-\alpha} h \right) \overline{\nabla_v \left( (t + b)^{-\alpha} h \right)} dvdt \right| \\
&\le C_m c_0 \int_0^{t_2} \int_{\mathbb{R}^m} A \nabla_v \left( (t + b)^{-\alpha} h \right) \overline{\nabla_v \left( (t + b)^{-\alpha} h \right)} dvdt.
\end{aligned}
\end{equation}
Combining \eqref{lemma01-6} and \eqref{lemma01-8}, and then using \eqref{condition-A}, gives the desired inequality.
\end{proof}

\bigskip 
\noindent
We now prove Lemma \ref{lemma02}. 
\begin{proof}[{\bf Proof of Lemma \ref{lemma02}}]
For simplicity, denote 
$$
J_2 :=  2Re \int_0^{t_2} \int_{\mathbb{R}^m} (t + b) i (B_1^T e^{-tB_2^T} \rho) \cdot v \left( (t + b)^{-\alpha} h \right) \overline{(\frac{\alpha}{t + b} + \Div_v \left( A \nabla_v \right)) \left( (t + b)^{-\alpha} h \right)} dvdt .
$$
Direct computation gives: 
$$
J_2 = 2Re \int_0^{t_2} \int_{\mathbb{R}^m} (t + b) i (B_1^T e^{-t B_2^T} \rho) \cdot v (t + b)^{-\alpha} h \overline{\Div_v \left( A \nabla_v \left( (t + b)^{-\alpha} h \right) \right)} dvdt, 
$$
where we use the fact that
$$
2\alpha Re \int_0^{t_2} \int_{\mathbb{R}^m} i (B_1^T e^{-t B_2^T} \rho)\cdot v |(t + b)^{-\alpha} h|^2 dvdt = 0.
$$
Integrating by parts in the $v$-variables, we obtain 
\begin{align*}
J_2 &= -2Re \int_0^{t_2} \int_{\mathbb{R}^m} (t + b) A \nabla_v \left[ i (B_1^T e^{-tB_2^T} \rho) \cdot v (t + b)^{-\alpha} h \right] \overline{\nabla_v \left( (t + b)^{-\alpha} h \right)} dvdt \\
&= -2Re \int_0^{t_2} \int_{\mathbb{R}^m} (t + b) i A (B_1^T e^{-tB_2^T} \rho) \cdot (t + b)^{-\alpha} h \overline{\nabla_v \left( (t + b)^{-\alpha} h \right)} dvdt, 
\end{align*}
where we used 
$$
-2Re\int_0^{t_2} \int_{\mathbb{R}^m} (t + b) i (B_1^T e^{-tB_2^T} \rho)\cdot v A \nabla_v \left( (t + b)^{-\alpha} h \right) \overline{\nabla_v \left( (t + b)^{-\alpha} h \right)} dvdt = 0.
$$
Using $|A| \le C_m \lambda$, the Cauchy--Schwarz inequality and Young's inequality, we obtain 
\begin{equation}\label{lemma02-8}
\begin{aligned}
|J_2| &\le C_m \int_0^{t_2} \int_{\mathbb{R}^m} (t + b) |B_1^T e^{-tB_2^T} \rho| \cdot |(t + b)^{-\alpha} h| \cdot |\nabla_v \left( (t + b)^{-\alpha} h \right)| dvdt \\
&\le \varepsilon \int_0^{t_2} \int_{\mathbb{R}^m} |\nabla_v \left( (t + b)^{-\alpha} h \right)|^2 dvdt + C_\varepsilon \int_0^{t_2} \int_{\mathbb{R}^m} (t + b)^2 |B_1^T e^{-tB_2^T} \rho|^2 |(t + b)^{-\alpha} h|^2 dvdt.
\end{aligned}
\end{equation}
By assumption \eqref{condi-1}, we have 
\begin{equation}\label{lemma02-9}
(t + b)^2 |B_1^T e^{-tB_2^T} \rho|^2 = (t + b)^{-1} (t + b)^3 |B_1^T e^{-tB_2^T} \rho|^2 \le \varepsilon_0 \alpha (t + b)^{-1}.
\end{equation}
Combining \eqref{lemma02-8} and \eqref{lemma02-9} proves the lemma.
\end{proof}

\bigskip\noindent {\bf Acknowledgements.}
This work was supported by the NSFC (No.12031006) and the Fundamental Research Funds for the Central Universities of China.


\bigskip
\begin{thebibliography}{99}

\bibitem{Bar-1} E. Barucci, S. Polidoro, V. Vespri, Some results on partial differential equations and Asian options, {\em Math. Models Methods Appl. Sci.} 11 (3) (2001) 475-497.

\bibitem{BS} F. Black, M. Scholes, The pricing of options and corporate liabilities, {\em J. Polit. Econ.} 81 (3) (1973) 637-654.


\bibitem{Car} T. Carleman, Sur un probl\`eme d'unicit\'e pour les syst\`emes d'\'equations aux d\'eriv\'ees partielles \`a deux variables ind\'ependantes, {\em Ark. Mat. Astr. Fys.}, 26 (17) (1939) 9.

\bibitem{o-3} H. Chen, W.-X. Li, C.-J. Xu, Gevrey hypoellipticity for linear and non-linear Fokker-Planck equations. {\em J. Differential Equations}, 246 (1) (2009) 320-339.

\bibitem{Des} L. Desvillettes, Plasma kinetic models: the Fokker-Planck-Landau equation, in: P. Degond, L. Pareschi, G. Russo (Eds.), {\em Modeling and Computational Methods for Kinetic Equations}, Model. Simul. Sci. Eng. Technol., Birkh\"auser Boston, Boston, MA, 2004, pp. 171-193.

\bibitem{ESS} L. Escauriaza, G. Seregin, V. \v{S}ver\'ak, Backward uniqueness for parabolic equations, {\em Arch. Rational Mech. Anal.}, 169 (2) (2003) 147-157.

\bibitem{ESS2} L. Escauriaza, G. Seregin, V. \v{S}ver\'ak, $L^{3, \infty}$ solutions to the Navier-Stokes
equations and backward uniqueness, {\em Russ. Math. Surv.}, 58 (2) (2003) 211-250.

\bibitem{Ghazaei} M. M. Ghazaei Ardakani, A. Robertsson, R. Johansson, Online minimum-jerk trajectory generation, in: {\em Proceedings of the IMA Conference on Mathematics of Robotics}, Oxford, United Kingdom, 2015.

\bibitem{GIMV} F. Golse, C. Imbert, C. Mouhot, A. F. Vasseur, Harnack inequality for kinetic Fokker-Planck equations with rough coefficients and application to the Landau equation, {\em Ann. Sc. Norm. Super. Pisa, Cl. Sci.}, (5) 19 (1) (2019) 253–295.

\bibitem{Imbert} C. Imbert, De Giorgi’s regularity theory for elliptic, parabolic and kinetic equations, arXiv preprint arXiv:2601.15238, 2026.

\bibitem{Kol-1} A. N. Kolmogorov, Zuf\"allige Bewegungen (zur Theorie der Brownschen Bewegung), {\em Ann. of Math.}, 35 (1) (1934) 116-117.

\bibitem{control-MM} K. Mehrotra, P. R. Mahapatra, A jerk model for tracking highly maneuvering targets, {\em IEEE Trans. Aerosp. Electron. Syst.}, 33 (4) (1997) 1094-1105.

\bibitem{MX-1} Y. Morimoto, C.-J. Xu, Ultra-analytic effect of Cauchy problem for a class of kinetic equations, {\em J. Differential Equations}, 247 (2) (2009) 596-617.

\bibitem{Mouhot01} C. Mouhot, De Giorgi--Nash--Moser and H\"ormander theories: new interplays, in: {\em Proceedings of the International Congress of Mathematicians--Rio de Janeiro 2018. Vol. III. Invited Lectures}, World Sci. Publ., Hackensack, NJ, 2018, pp. 2467-2493.

\bibitem{TAN-Z} Tu A. Nguyen, On a question of Landis and Oleinik, {\em Trans. Amer. Math. Soc.}, 362 (6) (2010) 2875-2899.

\bibitem{Palleschi} A. Palleschi, M. Garabini, D. Caporale, L. Pallottino, Time-optimal path tracking for jerk controlled robots, {\em IEEE Robot. Autom. Lett.}, 4 (4) (2019) 3932-3939.

\bibitem{fi-1} A. Pascucci, Kolmogorov equations in physics and in finance, {\em Prog. Nonlinear Differ. Equ. Their Appl.}, 63 (2005) 353-364.

\bibitem{Sonin-1} I. M. Sonin, On a class of degenerate diffusion processes, {\em Theory Probab. Appl.}, 12 (1967) 490-496.


\bibitem{Villani-1} C. Villani, On the spatially homogeneous Landau equation for Maxwellian molecules, {\em Math. Models Methods Appl. Sci.} 8 (6) (1998) 957-983.

\bibitem{C.W} C. Wang, Heat flow of harmonic maps whose gradients belong to $L^n_x L_t^\infty$, {\em Arch. Ration. Mech. Anal.}, 188 (2) (2008) 351-369.

\bibitem{Wang-Zhang01} W. Wang, L. Zhang, The $C^{\alpha}$ regularity of a class of non-homogeneous ultraparabolic equations, {\em Sci. China Ser. A.}, 52 (8) (2009) 1589-1606.

\bibitem{Wang-Zhang000} W. Wang, L. Zhang, Backward uniqueness of Kolmogorov operators, {\em Methods Appl. Anal.}, 20 (1) (2013) 79-88. 

\bibitem{WW} J. Wu, W. Wang, On backward uniqueness for the heat operator in cones, {\em J. Differential Equations}, 258 (1) (2015) 224-241.

\bibitem{W-Z01} J. Wu, L. Zhang, Backward uniqueness for parabolic operators with variable coefficients in a half space, {\em Commun. Contemp. Math.}, 18 (1) (2016) 1550011.

\bibitem{W-Z02} J. Wu, L. Zhang, Backward uniqueness for general parabolic operators in the whole space, {\em Calc. Var. Partial Differential Equations}, 58 (4) (2019) 155.

\bibitem{W-Z03} J. Wu, L. Zhang, The Landis--Oleinik conjecture in the exterior domain, {\em Adv. Math.}, 302 (2016) 190-230.

\bibitem{Xin-Zhang} Z. Xin, L. Zhang, On the global existence of solutions to the Prandtl’s system, {\em Adv. Math.}, 181 (1) (2004) 88-133.

\bibitem{Yamamoto} M. Yamamoto, Carleman estimates for parabolic equations and applications, {\em Inverse Probl.}, 25 (12) (2009) 123013.

\bibitem{Zhang000} L. Zhang, On the regularity of weak solutions for ultra-parabolic equations in the divergence form (in Chinese), {\em Sci. Sin. Math.}, 54 (3) (2024) 547-558.





\end{thebibliography}
\end{document}